\documentclass[11pt]{amsart}
\usepackage{graphicx}
\usepackage[active]{srcltx}
 \makeatletter
\renewcommand*\subjclass[2][2000]{%
  \def\@subjclass{#2}%
  \@ifundefined{subjclassname@#1}{%
    \ClassWarning{\@classname}{Unknown edition (#1) of Mathematics
      Subject Classification; using '1991'.}%
  }{%
    \@xp\let\@xp\subjclassname\csname subjclassname@#1\endcsname
  }%
}
 \makeatother
\usepackage{enumerate,amssymb,  mathrsfs,yhmath}%, pdfsync}% ,showkeys, pdfsync}

\newtheorem{theorem}{Theorem}[section]

\newtheorem*{lemma*}{Lemma}
\newtheorem{proposition}[theorem]{Proposition}

\def\1ton{1,2,\ldots,n}

\usepackage{amssymb}%, pdfsync}
\usepackage{amsthm}
\usepackage{mathrsfs, amsfonts, amsmath}
\usepackage{graphicx}

\theoremstyle{definition}

\theoremstyle{remark}
\newtheorem{remark}[theorem]{Remark}

\numberwithin{equation}{section}

 \DeclareMathOperator{\re}{Re}

\def\XXint#1#2#3{{\setbox0=\hbox{$#1{#2#3}{\int}$}
\vcenter{\hbox{$#2#3$}}\kern-.5\wd0}}

\def\ge{\geqslant}
\begin{document}
\title[Riesz and Kolmogorov inequality for harmonic mappings ]{Riesz and Kolmogorov inequality for harmonic quasiregular mappings }  \subjclass[2020]{Primary 30H10 }

%\date{11 October, 2005}

\author{David Kalaj}
\address{University of Montenegro, Faculty of Natural Sciences and
Mathematics, Cetinjski put b.b. 81000 Podgorica, Montenegro}
\email{davidk@ucg.ac.me}

%\author{Olivia Constantin }
%\address{University of Vienna, Department of Mathematics,  Oskar-Morgenstern-Platz 1, 1090 Wien}
%\email{olivia.constantin@univie.ac.at}

\begin{abstract}
Let $K\ge 1$ and $p\in(1,2]$.
We obtain asymptotically sharp constant $c(K,p)$, when $K\to 1$ in the inequality $$\|\Im f\|_{p}\le c(K,p)\|\Re(f)\|_p$$ where $f\in \mathbf{h}^p$ is a $K-$quasiregular harmonic mapping in the unit disk belonging to the Hardy space $\mathbf{h}^p$, under the conditions $\arg(f(0))\in (-\pi/(2p),\pi/(2p))$ and $f(\mathbb{D})\cap(-\infty,0)=\emptyset$. The paper improves a recent result by Liu and Zhu in \cite{aimzhu}.  We also extend this result for the quasiregular  harmonic mappings in the unit ball in $\mathbb{R}^n$. Finally, the Kolmogorov theorem for quasiregular harmonic mappings in the plane is considered.
\end{abstract}
\maketitle
\tableofcontents
\section{Introduction}
Let $\mathbb{D}$ denote the unit disk and $\mathbf{T}$ the unit circle in the complex plane.
For $p>1$, we define the Hardy class $\mathbf{h}^p$ as the class of harmonic mappings $f=g+\bar h$, where $g$ and $h$ are holomorphic mappings defined on the unit disk $\mathbb{D},$ so that $$\|f\|_p=\|f\|_{\mathbf{h}^p}=\sup_{0<r< 1} M_p(f,r)<\infty,$$ where $$M_p(f,r)=\left(\int_{\mathbf{T}}|f(r\zeta)|^p d\sigma(\zeta)\right)^{1/p}.$$ Here $d\sigma(\zeta)=\frac{dt}{2\pi},$ if $\zeta=e^{it}\in \mathbf{T}$. The subclass of holomorphic mappings that belongs to the class $\mathbf{h}^p$ is denoted by $H^p$.

If $f\in \mathbf{h}^p$, then it is well-known that there exists
$$f(e^{it})=\lim_{r\to 1} f(re^{it}),   a.e.$$  and $f\in L^p(\mathbf{T}).$
Then there hold
\begin{equation}\label{come}\|f\|^p_{{\mathbf{h}^p}}=\lim_{r\to 1}\int_{0}^{2\pi}|f(re^{it})|^p \frac{dt}{2\pi}= \int_{0}^{2\pi}|f(e^{it})|^p \frac{dt}{2\pi}.\end{equation}

Similarly, we define the {\it Hardy space} $\mathbf{h}^p$ of harmonic functions in the unit ball $\mathbb{B}\subset \mathbb{R}^n$. Let $\mathbb{S}=\partial \mathbb{B}$. Then we say a harmonic function $u$ in $\mathbb{B}$ belongs to the Hardy space $\mathbf{h}^p$ if $$\|f\|_p=\|f\|_{\mathbf{h}^p}=\sup_{0<r< 1} M_p(f,r)<\infty,$$ where $$M_p(f,r)=\left(\int_{\mathbb{S}}|f(r\zeta)|^p d\sigma(\zeta)\right)^{1/p}.$$ Here $d\sigma$ is the  surface $n-1$ dimensional
measure of the Euclidean sphere which satisfies the condition:
$\sigma(\mathbb{S})=1$. The Hardy space of harmonic function in the space has similar properties as in the plane (see e.g. \cite{ABR}).

Let $1<p<\infty$ and let $\overline p=\max\{p,p/(p-1)\}$. Verbitsky in \cite{ver} proved the following results.  If $f=u+iv\in H^p$ and $v(0)=0$, then  \begin{equation}\label{ver}\sec(\pi/(2\overline p))\|v\|_p\leq\|f\|_p,\end{equation}
and
\begin{equation}\label{1ver}\|f\|_p\leq\csc(\pi/(2\overline p))\|u\|_p,\end{equation}
and both estimates are sharp. Those results improve the sharp inequality \begin{equation}\label{pico}\|v\|_p\leq\cot(\pi/(2\overline p))\|u\|_p\end{equation} found by S.  Pichorides (\cite{pik}). For some related results see \cite{essen, studia, graf, verb2}.

Then those results have been extended by the author in \cite{tams}.  As a byproduct, the author proved a Hollenbeck-Verbitsky conjecture for the case $s=2$.

Further, those results have been extended by Markovi\'c and Melentijevi\'c in \cite{melmar} and Melentijevi\'c in \cite{mel}. Melentijevi\'c  proved  a Hollenbeck-Verbitsky conjecture for the case $s<\sec^2(\pi/(2p))$, $p\le 4/3$ or $p\ge 2$.

Recently Liu and Zhu in \cite{aimzhu} generalized Riesz conjugate functions theorem for
planar harmonic $K$-quasiregular mappings (when $1 < p \le
2$) provided that the real part does not vanish at the unit disk. Their result is asymptotically sharp when $K\to 1$.  Moreover, they extended this result for invariant harmonic quasiconformal mappings in the unit ball also assuming that the first coordinate is non-vanishing.

In this paper, we will remove the assumption that the real part of the mapping does not vanish. Moreover, our approach works for harmonic quasiregular mappings in the space.

\subsection{Quasiregular and harmonic mappings}
A continuous and nonconstant mapping $f : G\to \Bbb R^n$, $n\ge
2$, in the local Sobolev space $W^{1,n}_{loc} (G,\Bbb R^n)$ is
$K$-quasiregular, $K\ge 1$, if $$|Df(x)| \le K \ell(f' (x))$$ for almost
every $x\in G$, where $G$ is an open subset of $\mathbb{R}^n$. Here $Df(x)$ is the formal differential matrix and $$|Df(x)|=\sup_{|h|=1}|f'(x)h|,\ \  \ell(Df(x))=\inf_{|h|=1}|f'(x)h|.$$

Let $(Df(x))^t$ be the transpose of the matrix $Df(x)$. Then we define the Hilbert norm of $D f(x)$
by the formula $$\|Df(x)\|=\sqrt{\mathrm{trace}\left( Df^{t}Df\right)}=\sqrt{\sum_{k=1}^n\lambda_k},$$ where $\lambda_1\le \dots \le \lambda_n$ are eigenvalues of the matrix $Df^{t}Df$.

Observe that $$|Df^t(x)|=|Df(x)|, \ \  \ell(Df(x))=\ell (Df^t(x)).$$ Further if $f$ is $K$-quasiregular, then $$\nabla |f(x)|=\nabla \sqrt{\left<f(x), f(x)\right>}= Df^t(x) \frac{f(x)}{|f(x)|}.$$ Thus

\begin{equation}\label{inprof}
|\nabla |f(x)||\ge \ell(Df(x))\ge \frac{|Df(x)|}{K}.
\end{equation}
If $n=2$, then we use the notation  $k-$quasiconformal mapping for $k=(K-1)/(K+1)$.

A smooth mapping $w:G\to \mathbb{R}^n$ is called harmonic if it satisfies the Laplace equation $\Delta u=0$.
The solution of the equation $\Delta w=g$ (in the sense of
distributions see \cite{Her}) in the  ball $B_R=R\cdot \mathbb{B}$, satisfying the
boundary condition $w|_{S_R}=f\in L^1(S_R)$, where $S_R=R\cdot\mathbb{S}$ is given by
\begin{equation}\label{green}w(x)=\int_{S_R}P(x,\eta)f(\eta)d\sigma(\eta)-\int_{B_R}G(x,y)g(y)dV(y),\,
|x|<1.\end{equation} Here
\begin{equation}\label{poisson}P(x,\eta)=\frac{R^2-|x|^2}{R|x-\eta|^n}\end{equation}
is the Poisson kernel and $d\sigma$ is the  surface $n-1$ dimensional
measure of the Euclidean sphere which satisfies the condition:
$\int_{\mathbb{S}}d\sigma(\eta)\equiv 1$. The first integral
in (\ref{green}) is called the Poisson integral and is usually
denoted by $P[f](x)$. It is a harmonic mapping. The function
\begin{equation}\label{green1}G(x,y)=\left\{
                                       \begin{array}{ll}
                                         \frac{1}{2\pi}\log \frac{R|x-y|}{|R^2-x\overline{y}|}, & \hbox{for $n=2$;} \\
                                         c_n\left(\frac{1}{|x-y|^{n-2}}
-\frac{1}{(R^2+|x|^2|y|^2/R^2-2\left<x,y\right>)^{(n-2)/2}}\right), & \hbox{for $n\ge 3$,}
                                       \end{array}
                                     \right.\end{equation}
where
\begin{equation}\label{cen}c_n=\frac{1}{(n-2)\omega_{n-1}}\end{equation} and
$\omega_{n-1}$ is the measure of $\mathbb{S}$, is the Green function of
the  ball $B_R$. The Poisson kernel and the Green function are
harmonic in $x$.  %If $f\in L^1$ and $g\in L^\infty$ then $u-P[f]\in
%C^{1,\alpha}$; see \cite[Theorem~8.33]{gt}.

\section{Main results}
The main result of this paper are following theorems
\begin{theorem}\label{david} Let $$c_n^p(K,p)=\frac{(1+(n-1)K^2)(1 +\frac{(p-2)}{n K^2})}{(p-1)}.$$
\begin{enumerate}\item[a)] Assume that $f$ is a $K-$quasiregular harmonic mapping in the unit disk so that $\Im f(0)=0$. Assume that $\Re f\in \mathbf{h}^p$ for some $p\in (1,2]$. Then $f\in \mathbf{h}^p$ and we have the inequality
$$\|f\|_p\le c_2(K,p)\|\Re f\|_p.$$
\item[b)]
Assume that $f=(f_1,\dots, f_n)$ is a $K-$quasiregular harmonic mapping of the unit ball into $\mathbb{R}^n$. Assume that $f_1\in \mathbf{h}^p$ for some $p\in (1,2]$. Then $f\in \mathbf{h}^p$ and we have the inequality
$$\|f\|^p_p\le|f(0)|^p+  c_n(K,p)(\|f_1\|^p_p-|f_1(0)|^p).$$
\end{enumerate}
The constant $C_k(K,p)$ is asymptotically sharp when  $p\to 2$.

\end{theorem}
The following theorem is an improvement of the main result of Liu and Zhu \cite{aimzhu}.
\begin{theorem}\label{davidk}  Let $$c(p,K)=\left(\frac{\tan^{p-1}\frac{\pi}{2p}}{\cot(\frac{\pi}{2p})}+(K^2-1)\frac{\sin^{p-1}\frac{\pi}{2p}}{\cos \frac{\pi}{2p}}\right)^{1/p} ,$$ and $$d(p,K)=\left(\cos\left[\frac{\pi }{2 p}\right]^{-p} +(K^2-1)\tan\left[\frac{\pi }{2 p}\right]\right)^{1/p}.$$
\begin{enumerate}
\item[a)]
 Assume that $f$ is a $K-$quasiregular harmonic mapping in the unit disk so that $\arg(f(0))\in (-\pi/(2p), \pi/(2p))$. Assume that $\Re f\in \mathbf{h}^p$ for some $p\in (1,2]$ and that $f(\mathbb{D})\cap(-\infty,0)=\emptyset$. Then $f\in \mathbf{h}^p$ and we have the inequality
$$\|\Im f\|_p\le c(K,p)\|\Re f\|_p.$$ The constant $c(K,p)$ is asymptotically sharp.
\item[b)] Under condition of a)  we have the asymptotically sharp inequality $$\|f\|_p\le d(K,p)\|\Re f\|_p.$$
\end{enumerate}
\end{theorem}

We also have the following Kolmogorov type theorem
\begin{theorem}\label{terca}
Assume that $f=u+i v$ is harmonic $K-$quasiregular mapping in the unit disk $\mathbb{D}$ so that $u>0$ and let $0<p,r<1$ and that $v(0)=0$. Then  $u\in h^1$ and $v\in h^p$ and
%Then $f\in h^p$ and for every $r\in[0,1]$ we have the inequality
%\begin{equation}\label{njeb}
%M_p(r,f)\le D(p,K)M_1(r,u),
%\end{equation}
%where $$D^p(p,K)=2^{1-p}\left(1+\frac{\left(1+K^2\right) \left(1+\frac{-2+p}{2 K^2}\right)}{1-p}\right).$$
\begin{equation}\label{dyb}
M^p_p(r,v)\le \sec\frac{\pi p}{2} (K^2M^p_1(r,u)-(K^2-1)M^p_p(r,u))
\end{equation}
and
\begin{equation}\label{treb}
(2-K^2)M^p_1(r,u)\le (2-K^2) M^p_p(r,u) +\cos (p\pi/2) M^p_p(r,v).
\end{equation}
The constants in \eqref{dyb} and \eqref{treb} are asymptotically sharp when $K\to 1$.

\end{theorem}

\begin{remark}\label{vere}
Astala and Koskela
proved  in \cite[Theorem~6.1]{astala} that if $f$ is quasiconformal in $\mathbb{B}$ with one of its coordinate
functions belongs to $\mathbf{h}^p$, then $f\in h^q$ for all $q<p$. Furthermore, they showed that $q$ can not reach $p$ by giving a concrete example. Here we have a better outcome, but under essentially stronger conditions. The proofs of those theorems are similar to the proof of \cite[Theorem~4.1]{duren}, and this method has been also used by Liu and Zhu in \cite{aimzhu}. Theorem~\ref{davidk} improves the main result in \cite{aimzhu}, because we do not assume that $u$ is positive and get the same outcome.

Theorem~\ref{terca} is Kolmogorov  theorem for quasiregular harmonic mappings. We want to point out that it seems that inequality \eqref{treb} is new also for holomorphic functions.

\end{remark}

\section{Proof of main results}
\begin{proof}[Proof of Theorem~\ref{david}]
We will assume that the mapping $f$ has a smooth extension to the boundary. If not, then we take the dilatation $F(x)=f(r x)$, with $r<1$ and let $r\to 1$.

{\bf The case $n=2$}.
Since $$\|f\|_p^p=\frac{1}{2\pi}\int_{-\pi}^{\pi}|f(e^{it})|^p dt,$$ by using \eqref{green} to the function $w(z)=|f(z)|^p$ we get
$$|f(0)|^p= \int_{0}^{2\pi} |f(e^{it})|^p \frac{dt}{2\pi}-\frac{1}{2\pi}\int_{\mathbb{D}}\log \frac{1}{|w|} \Delta |f(w)|^p du dv.$$
Let  $u=\Re f$. Let $dA(z)=dxdy$. Then by \cite[Theorem~1.B]{pavlovic2}
\begin{equation}\label{pav}\|u\|_p^p = |u(0)|^p+\frac{p(p-1)}{2\pi }\int_{\mathbb{D}}|u|^{p-2} |\nabla u|^2 \log \frac{1}{|z|}dA(z).\end{equation}
Let $v(0)=0.$
 We first have $$\|f\|_p^p =|f(0)|^p+\frac{1}{2\pi }\int_{\mathbb{D}}\Delta|f(z)|^p\log\frac{1}{|z|}dA(z) $$
Since $$\Delta |f|^p=\frac{p(p-2)}{4}|f|^{p-4}\|\nabla |f|^2\|^2+\frac{p}{2}|f|^{p-2}\Delta|f|^2$$
and
$$ \frac{p}{2}|f|^{p-2}\Delta|f|^2=2p|f|^{p-2} (|g'(z)|^2+|h'(z)|^2),$$ and because  $f$ is $K-$quasiregular, by \eqref{inprof} we have $$|\nabla |f|^2|^2=4|f|^2|\nabla |f||^2\ge 4|f|^2 \|Df\|^2K^{-2}.$$
So \[\begin{split}\Delta |f|^p&\le \left(\frac{p(p-2)}{K^2} +2p \right)|f|^{p-2} (|g'(z)|^2+|h'(z)|^2)
\\&\le \left(\frac{p(p-2)}{K^2} +2p \right)|u|^{p-2} (|g'(z)|^2+|h'(z)|^2)
\\&\le \frac{1+K^2}{2}\left(\frac{p(p-2)}{K^2} +2p \right)|u|^{p-2} |g'(z)+h'(z)|^2.
\end{split}
\]
For $p\in(1,2)$ and $v(0)=0$ since $|\nabla u|^2 = |g'+h'|^2$, by the previous estimate we obtain that
\[\begin{split}\|f\|_p^p&\le
|f(0)|^p+\left(\frac{p(p-2)}{K^2} +2p  \right)\int_{\mathbb{D}}|f|^{p-2} (|g'(z)|^2+|h'(z)|^2) \log \frac{1}{|z|}\frac{dA(z)}{2\pi}
 \\& \le |f(0)|^p+ \frac{1+K^2}{2}\left(\frac{p(p-2)}{K^2} +2p \right)\int_{\mathbb{D}}|u|^{p-2} |\nabla u|^2 \log \frac{1}{|z|}\frac{dA(z)}{2\pi}
\\&=|f(0)|^p+ \frac{1+K^2}{2}\left(\frac{p(p-2)}{K^2} +2p \right)\frac{1}{p(p-1)}(\|u\|_p^p-|u(0)|^p)
\\&\le \left(\frac{(p-2)}{K^2} +2 \right)\frac{1+K^2}{2(p-1)} \|u\|_p^p.\end{split}
\]

{\bf The case $n>2$}.
This proof is similar to the proof of the case $n=2$. The only difference appears due to the different Green functions.
Assume that $u>0$ and assume that $\omega_n$ is the $n-1$ dimensional area of $\mathbb{S}$ and let $c_n=1/(\omega_n(n-2))$.
Then by \eqref{green} we have 
$$\|f\|_p^p =|f(0)|^p+c_n\int_{\mathbb{B}}\Delta|f(x)|^p\left(|x|^{2-n}-1\right)dV(x).$$
Further, we have (see e.g. \cite{jmaa})
\[\begin{split}\Delta |f|^p&=\frac{p(p-2)}{4}|f|^{p-4}\|\nabla |f|^2\|^2+\frac{p}{2}|f|^{p-2}\Delta|f|^2
\\&=\frac{p(p-2)}{4}|f|^{p-4}\|\nabla |f|^2\|^2+p|f|^{p-2}\|Df\|^2.\end{split}\]
Since $$|Df|^2=\max\{{\lambda_k}, k\in\{1,\dots, n\}\}\ge\frac{1}{n}\sum_{k=1}^m\lambda_k=\frac{1}{n}\|Df\|^2,$$ where $\lambda_k$ are the eigenvalues of $(Df)^\ast Df$, in view of \eqref{inprof}, we get that

$$\|\nabla |f|^2\|=4 |f|^2\|\nabla |f|\|^2\ge 4|f|^2 \frac{|Df|^2}{K^2}\ge 4|f|^2 \frac{\|Df\|^2}{nK^2}.$$
Thus $$\Delta |f|^p\le (p +\frac{p(p-2)}{n K^2})|f|^{p-2}\|Df\|^2.$$
On the other let $u=f_1$. Then we have $$\Delta |u|^p=p(p-1)|\nabla u|^2|u|^{p-2}$$
Thus $$\Delta |f|^p\le\frac{(1+(n-1)K^2)(1 +\frac{(p-2)}{n K^2})}{(p-1)}\Delta u^p. $$
Now we use \cite[Theorem~4.3]{pavlovic2} which states the following $$\|u\|_p^p=|u(0)|^p+\frac{p(p-1)}{n(n-2)}\int_{\mathbb{B}}|u(x)|^{p-2}|\nabla u(x)|^2(|x|^{2-n}-1)dV(x).$$
Then for $$C(K,p)=\frac{(1+(n-1)K^2)(1 +\frac{(p-2)}{n K^2})}{(p-1)}$$
\[
\begin{split}
\|f\|_p^p&\le|f(0)|^p+c_n\int_{\mathbb{B}}\Delta |f|^p  \left(|x|^{2-n}-1\right)dV(x)
\\&\le |f(0)|^p+C(K,p)c_n\int_{\mathbb{D}}\Delta |u|^p \left(|x|^{2-n}-1\right)dV(x)
\\&\le |f(0)|^p+C(K,p)\left(\|u\|_p^p-|u(0)|^p\right).
\end{split}
\]

To prove the last statement of the theorem, choose $I(x)=x$ and consider first the case $K=1$. Then the equality holds in $\|I\|_2=C_2(1,2)\|I_1\|_2$, where $I_1(x)=x_1$. Here $C_2(1,2)=\sqrt{n}$.
In order to see this observe that $$\frac{1}{\omega_n}\int_{\mathbb{S}} x_1^2 d\sigma(x)=\frac{1}{\omega_n}\int_{\mathbb{S}} x_2^2 d\sigma(x)=\dots=\frac{1}{\omega_n}\int_{\mathbb{S}} x_n^2 d\sigma(x).$$ So summing altogether we get $$\frac{1}{\omega_n}\int_{\mathbb{S}} x_1^2 d\sigma(x)=\frac{1}{n}.$$  In a similar way we prove that  $f(x) =(x_1,K x_2,\dots, K x_n)$  is $K-$quasiconformal harmonic and we have $$\frac{\|f\|_2^2}{\|f_1\|_2^2}=C(K,2)=(1+(n-1)K^2).$$
\end{proof}

\subsection{Power of a quasiregular function}

Observe that as in Remark~\ref{vere},  $\Re(f^p(z))=|f(z)|^p \cos (p\arg (f(z)))$ is smooth in $\mathbb{D}$, because $\arg: \mathbb{C}\setminus (-\infty,0]\to (-\pi,\pi)$ is smooth.  Then after straightforward calculation, we get
$$\Delta (f^p) =p(p-1) f^{p-2}(f_x^2+f_y^2).$$ Then $$\Delta \Re (f^p) =p(p-1) \Re(f^{p-2}(f_x^2+f_y^2)).$$  Now we have $$f_x =g' +\overline{h'}$$  and $$f_y =i(g'-\overline{h'}).$$
So $$f_x^2+f_y^2 = 4 g'\overline{h'}.$$
Thus for $R=|f|$, $$|\Delta \Re (f^p)|\le p(p-1) R^{p-2}(|f_x^2+f_y^2|=4p(p-1)R^{p-2}|g'h'|.$$
Therefore  \begin{equation}\label{fut}|\Delta \Re (f^p)|\le 4 p(p-1) R^{p-2}|g'\cdot h'|\le 4 k p(p-1) R^{p-2}| g'|^2. \end{equation}
\begin{proof}[Proof of Theorem~\ref{davidk}]
As in Theorem~\ref{david}, we will assume that the mapping $f$ has a smooth extension to the boundary. If not, then we take the dilatation $F(z)=f(r z)$, with $r<1$ and let $r\to 1$.

a)
For $p\in[1,2]$ by \cite[Lemma 2.1]{pik}, we have for $p\in[1,2]$ and $|x|\le \pi$, \begin{equation}\label{cont}|\sin x|^p \le  A(p) |\cos x|^p-B(p)\cos(px),\end{equation} where $$A(p)=\frac{\tan^{p-1}\frac{\pi}{2p}}{\cot(\frac{\pi}{2p})},  \ \ B(p)=\frac{\sin^{p-1}\frac{\pi}{2p}}{\cos \frac{\pi}{2p}} .$$ Notice that the corresponding inequality in \cite{pik} is formulated only for $|x|\le \pi/2$, but the inequality \eqref{cont} continues to hold for $|x|\le \pi$. Indeed, just observe that if $x\in[\pi/2,\pi]$, then $y=\pi-x\in[0,\pi/2]$ and \eqref{cont} hold for $y$ instead of $x$. But we also have $\cos (p(\pi-x))\ge \cos px$ for $x\in[\pi/2,\pi]$. This proves the claimed inequality.

Thus $$|v(r e^{it})|^p\le A(p)|u(r e^{it})|^p-B(p)\Re(f^p(r e^{it})).$$

In particular \begin{equation}\label{vf}\|v\|_p^p\le A(p)\|u\|_p^p-B(p)\int_{-\pi}^{\pi}\Re(f^p( e^{it}))\frac{dt}{2\pi}.\end{equation}
Further, because $\Re f^p$ is smooth, $$\int_{-\pi}^{\pi}\Re(f^p(( e^{it})))\frac{dt}{2\pi}= \Re(f^p(0))+\frac{1}{2\pi}\int_{\mathbb{D}}\Delta \Re f^p(w) \log\frac{1}{|w|}du dv,$$ and
$$\int_{-\pi}^{\pi}|u( e^{it})|^p\frac{dt}{2\pi}= |u(0)|^p+\frac{1}{2\pi}\int_{\mathbb{D}}\Delta |u^p(w)| \log\frac{1}{|w|}du dv.$$
Further, because $u=\Re (g+h)$, we have \begin{equation}\label{equ}\begin{split}\Delta |u|^p&=p(p-1)|\nabla u|^2|u|^{p-2}\\&=p(p-1)|g'+h'|^2|u|^{p-2}\\&\ge p(p-1)(1-k)^2|g'|^2 |u|^{p-2}.\end{split}\end{equation}

Then \eqref{fut} and \eqref{equ} imply that $$|\Delta \Re (f^p)|\le \frac{4k}{(1-k)^2}\Delta |u|^p.$$

Thus \[\begin{split}\int_{-\pi}^{\pi}\Re(f^p( e^{it}))\frac{dt}{2\pi}&=
\Re(f^p(0))+\frac{1}{2\pi}\int_{\mathbb{D}}\Delta \Re f^p(w) \log\frac{1}{|w|}du dv,
\\& \ge  \Re(f^p(0))-\frac{1}{2\pi}\int_{\mathbb{D}}|\Delta \Re f^p(w)| \log\frac{1}{|w|}du dv\\\
&\ge
 \Re(f^p(0))-\frac{4k}{(1-k)^2}\frac{1}{2\pi}\int_{\mathbb{D}}\Delta| u^p(w)| \log\frac{1}{|w|}du dv\\
&=\Re(f^p(0))-\frac{4k}{(1-k)^2}\left(-|u(0)|^p+\int_{-\pi}^{\pi}|u( e^{it})|^p\frac{dt}{2\pi}\right)
\\&=\Re(f^p(0))+\frac{4k}{(1-k)^2}|u(0)|^p- \frac{4k}{(1-k)^2}\int_{-\pi}^{\pi}|u( e^{it})|^p\frac{dt}{2\pi}.
 \end{split}\]
Now if $\theta =\arg (f(0))\in (-\pi/(2p), \pi/(2p)$  then $$\Re(f^p(0))+\frac{4k}{(1-k)^2}|u(0)|^p\ge |f(0)|^p\cos (\theta p)\ge 0.$$
Thus $$\int_{-\pi}^{\pi}\Re(f^p(( e^{it})))\frac{dt}{2\pi}\ge -\frac{4k}{(1-k)^2}\int_{-\pi}^{\pi}|u( e^{it})|^p\frac{dt}{2\pi}.$$
Hence
\[\begin{split}\|v\|_p^p&\le A(p)\|u\|_p^p-B(p)\int_{-\pi}^{\pi}\Re(f^p(( e^{it})))\frac{dt}{2\pi}
\\ &\le   A(p)\|u\|_p^p+B(p) \frac{4k}{(1-k)^2}\|u\|_p^p
\\&= \left(A(p)+ (K^2-1)B(p)\right)\|u\|_p^p.
\end{split}\]
b) In this case we use the following inequality (\cite{ver}): for every $t\in[-\pi,\pi]$ we have
$$-1+\cos\left[\frac{\pi }{2 p}\right]^{-p} |\cos t|^p -\cos (pt) \tan\left[\frac{\pi }{2 p}\right]\ge 0.$$
Then the proof is very similar to the previous proof, but instead of \eqref{vf}
we use the inequality \begin{equation}\label{vf1}\|f\|_p^p\le C(p)\|u\|_p^p-D(p)\int_{-\pi}^{\pi}\Re(f^p(( e^{it})))\frac{dt}{2\pi},\end{equation} where $$C(p)=\cos\left[\frac{\pi }{2 p}\right]^{-p} \text{ and }D(p)=\tan\left[\frac{\pi }{2 p}\right].$$
\end{proof}
\begin{remark}
The proof of Theorem~\ref{davidk} also works under more general condition $$\cos (p \theta)+\frac{4k}{(1-k)^2} |\cos \theta|^p\ge 0,$$ where $\theta=\arg(f(0))\in(-\pi, \pi]$. %We believe that a method used by Pichorides in \cite{pik} can be of help for a general solution of this problem. Namely Pichorides constructed a subharmonic function $F(z)=|z|^p\cos(p\arctan\frac{y}{|x|})$.
\end{remark}

\begin{proof}[Proof of Theorem~\ref{terca}]
Notice that in this case (for positive $u$) the means $$M_p(r, u)=\left(\int_{0}^{2\pi}u^p(r e^{it})\frac{dt}{2\pi}\right)^{1/p}$$ are decreasing (\cite[Sec.~3.4]{pavlovic}) and thus $\|u\|_p=u(0)$ for every $0<p<1$.
As in proof of Theorem~\ref{davidk} we have \begin{equation}\label{fut}|\Delta \Re (f^p)|\le 4 p(1-p) r^{p-2}|g'\cdot h'|\le 4 k p(1-p) r^{p-2}| g'|^2. \end{equation}
\begin{equation}\label{equ1}\begin{split}\Delta |u|^p&=p(p-1)|\nabla u|^2|u|^{p-2}\\&=p(1-p)|g'+h'|^2|u|^{p-2}.\end{split}\end{equation}
Hence $$-\Delta |u|^p\ge p(1-p)(1-k)^2|g'|^2 |u|^{p-2}.$$
Thus
$$|\Delta \Re (f^p)|\le -\frac{4k}{(1-k)^2}\Delta|u|^p.$$
Now by \eqref{green} $$\int_0^{2\pi} u^p(re^{it}) \frac{dt}{2\pi}=u^p(0)+\frac{1}{2\pi }\int_{\mathbb{D}_r} \Delta u^p\log \frac{r}{|z|} dx dy,$$ where $\mathbb{D}_r=\{rz, |z|<1\}$.
By using again \eqref{green} $$\int_0^{2\pi}  \Re f^p(re^{it}) \frac{dt}{2\pi}=\Re f^p(0)+\frac{1}{2\pi }\int_{\mathbb{D}_r} \Delta \Re f^p\log \frac{r}{|z|} dx dy.$$
Thus
\begin{equation}\label{sh1}\begin{split}\int_0^{2\pi}  \Re f^p(re^{it}) \frac{dt}{2\pi}&\ge \Re f^p(0)+\frac{4k}{(1-k)^2}\frac{1}{2\pi }\int_{\mathbb{D}_r} \Delta |u|^p\log \frac{r}{|z|} dx dy
\\&=\Re f^p(0)+\frac{4k}{(1-k)^2}\left(\int_0^{2\pi} u^p(re^{it}) \frac{dt}{2\pi}-u^p(0)\right)
\\&=u^p(0)+\frac{4k}{(1-k)^2}\left(\int_0^{2\pi} u^p(re^{it}) \frac{dt}{2\pi}-u^p(0)\right)\end{split}\end{equation}
and
\begin{equation}\label{sh2}\begin{split}\int_0^{2\pi}  \Re f^p(re^{it}) \frac{dt}{2\pi}&\le \Re f^p(0)-\frac{4k}{(1-k)^2}\frac{1}{2\pi }\int_{\mathbb{D}_r} \Delta |u|^p\log \frac{r}{|z|} dx dy
\\&=\Re f^p(0)-\frac{4k}{(1-k)^2}\left(\int_0^{2\pi} u^p(re^{it}) \frac{dt}{2\pi}-u^p(0)\right)
\\&= (u(0))^p-\frac{4k}{(1-k)^2}\left(\int_0^{2\pi} u^p(re^{it}) \frac{dt}{2\pi}-u^p(0)\right),\end{split}\end{equation} because $v(0)=0$.
Since $u$ is harmonic, by mean value property we have $$\int_0^{2\pi}  \Re f^p(re^{it}) \frac{dt}{2\pi}\le \left(\int_0^{2\pi}  u(re^{it}) \frac{dt}{2\pi}\right)^p+\frac{4k}{(1-k)^2}\left(u^p(0)-\int_0^{2\pi} u^p(re^{it}) \frac{dt}{2\pi}\right).$$
Further by \cite[Eq.~22]{pik}, $$ \cos p x\ge \cos (p/2\pi)|\sin x|^p$$ and thus for $f(re^{it})=R e^{ix}$, $$\Re f^p(re^{it}) \ge \cos\frac{\pi p}{2}|\Im f(re^{it})|^p.$$
Hence by \eqref{sh2}, $$M^p_p(r,v)\le \sec\frac{\pi p}{2} (K^2 M_1^p(r,u)-(K^2-1)M_p^p(r,u)) .$$
Moreover, by using again \cite[Eq.~22]{pik}  $$\cos (px)\le |\cos x|^p+\cos (p\pi/2)|\sin x|^p,$$ for every $|x|\le \pi$. Thus by using the previous inequality to $f(re^{it})=R e^{ix}$, we have $$\Re f^p(re^{it})\le |u(re^{it})|^p+\cos (p\pi/2)|v(re^{it})|^p.$$
Hence, in view or \eqref{sh1},
$$u^p(0)+\frac{4k}{(1-k)^2}\left(\int_0^{2\pi} u^p(re^{it}) \frac{dt}{2\pi}-u^p(0)\right)\le M^p_p(r,u)+\cos (p\pi/2)M^p_p(r,v).$$
And therefore $$(2-K^2)M^p_1(r,u)\le (2-K^2) M^p_p(r,u) +\cos (p\pi/2) M^p_p(r,v),$$ because $u(0)=M_1(r,u)$ for every $r$.

Observe that for $K=1$, the constant in \eqref{dyb} coincides with the corresponding sharp constant in Kolmogorov theorem (\cite[Theorem~4.2]{duren}). Similarly the constant in \eqref{treb} is sharp. In this case we make use of $f(z)=\frac{1+z}{1-z}$. Then the equality is attained in \eqref{treb} for $K=1$ and for the cases $r=0$ and $r=1$. The case $r=0$ is easy. For $r=1$ we use the formula $$2\int_{0}^{\pi}\cot^{p}\frac{t}{2}dt=2\pi \sec \frac{\pi p}{2}$$ to obtain the equality.
\end{proof}
\begin{remark} We expect the positiveness of real part $u$ of $f$ in Theorem~\ref{terca} is redundant. We want to mention that it is an open question whether  the constant $\sec \frac{\pi p}{2}$ is sharp in \eqref{dyb}, for the class of holomorphic mappings without positiveness condition on the real part (\cite{pik}). On the other hand, the validity of Kolmogorov theorem for quasiregular harmonic mappings for a certain constant which is not as good as in Theorem~\ref{terca} can be proved.  We use the so called (Littlewood-Paley) $g$-function to prove a more general result. The idea of using $g-$function for this purpose but for $p>1$ appears in a recent preprint by Chen and Huang  \cite{chen}.
\end{remark}
\begin{proposition}\label{propo}
Let $f=g+\bar h$ be a $K-$quasiregular harmonic mapping, and assume that $0<p<1$ and that $\Im f(0)=0$.  Assume also that $\re f\in h^1$. Then $f\in h^p$ and there is a constant $C_0(K,p)$ so that \begin{equation}\label{pavlovi}\|f\|_p\le C_0(K,p)\|\Re f\|_1.\end{equation}
\end{proposition}
\begin{proof} 
 Note first that if $H$ is an analytic function defined in the unit disk with $H(0)=0$, then we have the double inequality $$\|H\|_p\le c_1(p)\|G[H]\|_p\le c_2(p)\|H\|_p,$$ where \begin{equation}\label{cand}G[H](z)=\left(\int_{0}^1|H'(r z)|^2(1-r)dr\right)^{1/2}.\end{equation} This hold due to Calderon type theorem of Pavlovi\' c \cite[Theorem~10.6]{pavlovic}.
Assume first that $g(0)=h(0)=0$. 
%Further by Kolmogorov inequality we have $$\|H\|_p\le C(p)(|H(0)|+\|\Re H\|_1),$$ if $\Re H\in h^1$. 
Let $F= g+h$. Then $\Re F=\Re f$.

Since $f$ is quasiregular we have $$|g'(z)|\le \frac{1+K}{2}|F'(z)|, \ \ z\in\mathbb{D}.$$
Then by using  \eqref{cand}, and classical Kolmogorov theorem, we have \[\begin{split}\|g\|_p&\le  c_1(p)\|G[g]\|_p
\\&\le \frac{1+K}{2}c_1(p)\|G[F]\|_p
\\&\le \frac{1+K}{2}c_1(p)\cdot c_2(p)\|F\|_p
\\&\le\frac{1+K}{2}c_3(p)\|\Re F\|_1=C_1(K,p)\|\Re F\|_1.\end{split}\]
Similarly we obtain $$\|h\|_p\le C_1(K,p)\|\Re F\|_1.$$
Since $\|f\|^p_p\le \|g\|^p_p+\|h\|^p_p$, we obtain  \eqref{pavlovi} with $C_0(K,p)=C(K,p):=2^{1/p} C_1(K,p)$.
If $g(0)\neq 0$ or $h(0)\neq 0$, then  we use the inequalities 
\[\begin{split}\|f\|^p_p&\le\|f-f(0)\|^p_p+\|f(0)\|^p_p\\&= \|f-f(0)\|^p_p+|\Re f(0)|^p\\&\le C^p(K,p)\|\Re(f-f(0))\|^p_1+|\Re f(0)|^p\le (2C^p(K,p)+1)\|\Re f\|^p_1,\end{split}\] because $|\Re f(0)|\le \|\Re f\|_1$ and $f(0)=\Re f(0)$. In this case $C_0(K,p)=(2C^p(K,p)+1)^{1/p}$.
\end{proof}
\subsection*{Acknowledgments} I would like to thank Professor Shaolin Chen for the useful discussion on this topic and sharing an earlier version of the manuscript \cite{chen}.

\end{document}